\documentclass[11pt]{article}
\usepackage{amsmath}
\usepackage{amsthm}
\usepackage{amssymb}
\usepackage{hyperref}

\newtheorem{theorem}{Theorem}[section]
\newtheorem{lemma}[theorem]{Lemma}
\newtheorem{corollary}[theorem]{Corollary}
\newtheorem*{question}{Question}
\newtheorem*{definition}{Definition}
\newtheorem*{problem}{Universal set problem}

\title{\vspace{-0.7cm}Discrete Kakeya-type problems and small bases}

\author{Noga Alon\thanks{Schools of Mathematics and Computer Science,
Raymond and Beverly Sackler Faculty of Exact Sciences, Tel Aviv
University, Tel Aviv 69978, Israel and IAS, Princeton, NJ 08540,
USA. Email: {\tt nogaa@tau.ac.il}. Research supported in part by the
Israel Science Foundation, by a USA-Israeli BSF grant,
by the Hermann Minkowski Minerva Center
for Geometry at Tel Aviv University and by the Von Neumann Fund.} \and
Boris Bukh\thanks{Department of Mathematics, Princeton University, Princeton,
NJ 08544, USA. Email: {\tt bbukh@math.princeton.edu}.}
\and Benny Sudakov\thanks{Department of Mathematics,
UCLA,  Los Angeles, CA 90095 and Institute for Advanced Study, Princeton,
NJ. Email: {\tt bsudakov@math.ucla.edu}.
Research supported in part by NSF CAREER award DMS-0546523, NSF
grants DMS-0355497 and DMS-0635607, by a USA-Israeli BSF grant, and by the
State of New Jersey.}}

\newcommand*{\R}{\mathbb{R}}
\newcommand*{\Z}{\mathbb{Z}}
\newcommand*{\abs}[1]{\lvert #1\rvert}

\date{}

\begin{document}
\maketitle
\begin{abstract}
A subset $U$ of a group $G$ is called $k$-universal if $U$ contains
a translate of every $k$-element subset of $G$. We give several
nearly optimal constructions of small $k$-universal sets, and use them to resolve an
old question of
Erd\H{o}s and Newman on bases for sets of integers, and to obtain
several extensions for other groups.
\end{abstract}
\section{Introduction}
A subset $U$ of $\R^d$ is a Besicovitch set if it contains a unit-length
line segment in every direction.
The Kakeya problem asks for the smallest possible Minkowski dimension of
a Besicovitch set. It is widely conjectured that every Besicovitch
set has Minkowski dimension~$d$. For large $d$ the best lower bounds come from
the approach pioneered by Bourgain \cite{bourgain_kakeya_arith}
which is based on combinatorial number theory. For example, in
\cite{bourgain_kakeya_montg} it is shown
that if every set $X\subset \Z/p\Z$ containing a translate
of every $k$-term arithmetic progression is of size at least
$\Omega(p^{1-\epsilon(k)})$
with $\epsilon(k)\to 0$ as $k\to\infty$, then the Kakeya conjecture
is true.

In this paper we address a related problem where instead of
seeking a set containing a translate of every $k$-term arithmetic
progression, we demand that the set contains a translate
of every $k$-element set. We do not restrict the problem to the
cyclic groups, and consider general (possibly non-abelian) groups.
Given a finite group $G$, we call a subset
$U$ of $G$ \emph{$k$-universal} if $U$
contains a left translate of every $k$-element set.
More generally, we say that $U$ is $k$-universal for $X\subset G$,
if for any $k$-element set $W=\{w_1,\dotsc,w_k\}\subset X$
there is a $g\in G$ such that $gW=\{gw_1,\dotsc,gw_k\}$ is contained in $U$.
We are interested in small $k$-universal sets for $k>1$
(for $k=1$ any one-element set is universal).

Since $U$ contains $\binom{\abs{U}}{k}$
$k$-element subsets, and the orbit of a $k$-set under (left) multiplication
by $G$ has size at most $\abs{G}$, it follows that
$\binom{\abs{U}}{k}\geq \binom{\abs{G}}{k}/\abs{G}$ for any
$U$ that is $k$-universal for the group~$G$.
>From that it is easy to deduce that $\abs{U}\geq \frac{1}{2}\abs{G}^{1-1/k}$.
%U^k>=U(U-1)..(U-k+1)>=(G-1)...(G-k+1)>=(G-k+1)^{k-1}
%if k<=|G|/2
%U>=(|G|/2)^{1-1/k}>=(1/2)|G|^{1-1/k}
%if k>|G|/2 then U has to contain k elements
%|U|>=|G|/2>=(1/2)|G|{1-1/k}
It is natural to wonder how sharp this lower bound on the size
of $k$-universal sets is.
\begin{question}
Is it true that for every integer $k>1$ there is a constant $c(k)$ such that
every finite group $G$ contains a $k$-universal set of size
not exceeding $c(k)\abs{G}^{1-1/k}$?   If so, is there a universal constant $c$
such that $c(k) \leq c$ for all $k$?
\end{question}
Efficient constructions of $k$-universal sets are hard to come by.
For $k=2$ the problem was solved by Kozma and Lev \cite{kozma_lev}
and independently by Finkelstein, Kleitman and Leighton \cite{fkl_vlsi}
who showed that the easy lower bound above is tight.
\begin{theorem}
For every finite group $G$ there is a $2$-universal set $U\subset G$
of size not exceeding $c\abs{G}^{1/2}$, where $c$ is an
absolute constant.
\end{theorem}

Their proofs relied heavily on the classification of finite simple groups,
and do not seem to generalize to constructions of $k$-universal sets
for $k\geq 3$. However, an easy probabilistic argument produces
$k$-universal sets with a loss of a logarithmic factor.
To state the actual result, which is a bit more general, we first need to define
an auxiliary notion of a non-doubling set in a group. This notion
is a special case of the general concept of sets with small doubling,
studied in combinatorial number theory.
\begin{definition}
A set $X\subset G$ is \emph{non-doubling} if $XX=\{xx' : x,x'\in X\}$
has at most $3\abs{X}$ elements.
\end{definition}
Note that in particular, any subgroup of a group
$G$ is non-doubling.
\begin{theorem}
\label{randthm}
For every non-doubling set $X$ of size $|X|>1$ in a finite group $G$ there is a $U\subset G$
which is $k$-universal for $X$,
of size  $\abs{U}\leq 36\abs{X}^{1-1/k}\log^{1/k}\abs{X}$.
In particular, for every finite group $G$ there is a $k$-universal
set $U$ of size $\abs{U}\leq 36\abs{G}^{1-1/k}\log^{1/k}\abs{G}$.
\end{theorem}
An interesting aspect of this result is that it gets better as $k$
gets larger. For instance, if $k\geq \log\log\abs{G}$, then
$(\log\abs{G})^{1/k}=O(1)$, showing
that the lower bound is tight up to a constant factor for moderately
large values of $k$. Here and everywhere in the paper
the logarithms are natural  (to the base $e=2.71\dotsc$).
For simplicity of presentation we also assume,  throughout the paper,
that all groups considered here are sufficiently large.

It is possible to improve upon the probabilistic construction
when the group is cyclic for any value of $k$, as well
as for some large families of groups for fixed~$k$.
\begin{theorem}
\label{constrthm}
\begin{enumerate}
\item
\label{cyclicconstr} If $G$ is cyclic, there is a $k$-universal set
of order at most $72\abs{G}^{1-1/k}$.
\item \label{symconstr} If $G=S_n$ is a symmetric group,
there is a $k$-universal $U\subset G$ of size at most
$(3k+1)!\abs{G}^{1-1/k}$.
\item \label{abelconstr} If $G$ is Abelian, there is a $k$-universal
$U\subset G$ of size at most $8^{k-1} k \abs{G}^{1-1/k}$.
\end{enumerate}
\end{theorem}

\noindent
Moreover, many more families of groups containing $k$-universal
sets of size $c(k)\abs{G}^{1-1/k}$ can be constructed
by using lemma~\ref{subgrouplem} below.

We will apply these results to settle an old problem of Erd\H{o}s and Newman
\cite{erdos_newman} on bases for sets of integers.
They studied bases for $m$-element subsets
of $\{1,\dotsc,n\}$, where a set $B$ is a basis for $A$
if $A\subset B+B=\{b_1+b_2 : b_1,b_2\in B\}$. Since
$\{0\}\cup A$ is a basis for $A$, and
there is a set $X$ with at most $c\sqrt{n}$ elements such that
$X+X\supset\{1,\dotsc,n\}$ it follows that
for any $m$-element subset of $\{1,\dotsc,n\}$
there is always a basis of size $\min(c\sqrt{n},m+1)$.
Erd\H{o}s and Newman showed by a counting argument,
that compared the number of $m$-element sets with the
number of possible bases of a given size,
that if $m$ is somewhat smaller than $\sqrt{n}$,
say $m=O(n^{1/2-\epsilon})$, then almost no $m$-element set
has a basis of size $o(m)$. Similarly, if $m$ is at
least $n^{1/2+\epsilon}$ almost all $m$-element sets require
a basis of size at least $c\sqrt{n}$. However,
for the borderline case
when $m$ is of the order $\sqrt{n}$ the
counting argument only yielded existence
of sets that need a basis of size
$c\sqrt{n}\log\log n/\log n$. They asked
if every $m$-set of size $m=\sqrt{n}$ has
a basis with $o(m)$ elements. We answer their question in
the affirmative not only for subsets of $\{1,\dotsc,n\}$
but in a much greater generality that applies to many other groups.

A subset $B$ of a group $G$
is said to be a basis for $A\subset G$ if $A\subset BB=\{bb' : b,b'\in B\}$.
Though the case when $G=(\Z,+)$ and $A\subset\{1,\dotsc,n\}$
is the setting in which Erd\H{o}s and Newman asked
their question, it is better to think of
their question in the group
$G=\Z/n\Z$. On one hand, if $A,B\subset \Z$ and $B+B\supset A$,
then the sets $B'=(B\mod n)\subset \Z/n\Z$ and $A'=(A\mod n)\subset\Z/n\Z$
satisfy $B'+B'\supset A'$. On the other hand,
if $A',B'\subset \Z/n\Z$ satisfy $B'+B'\supset A'$, then
thinking of $A'$ and $B'$ as subsets of $\{1,\dotsc,n\}$
in the natural way and letting $B=B'\cup(B'-n)$ we have
$B+B\supset A'$ in the group of integers.
So, up to a multiplicative constant of $2$ the
Erd\H{o}s-Newman problem is a problem about bases for subsets
of $\Z/n\Z$.

The lower bound of Erd\H{o}s and Newman immediately
carries over to any finite group $G$: there
is  always a set $A$ with at most $\sqrt{\abs{G}}$ elements
for which every basis is of size
at least $c\sqrt{G}\log\log\abs{G}/\log\abs{G}$. It turns out that
this bound is tight for many groups including the cyclic groups.
We say that the group $G$ of order $n$ satisfies the EN-condition if for every
$A\subset G$ of size at most $\sqrt{n}$ there is a basis $B$ of size
$\abs{B}\leq 50\frac{\sqrt{n}\log\log n}{\log n}$.
\begin{theorem}
\label{nondoublingENthm}
If $|G|=n$ and $G$ contains a non-doubling set $X$ satisfying $\sqrt{n}\log^2 n
\leq \abs{X}\leq\sqrt{n}\log^{10}n$, then $G$ satisfies the EN-condition.
\end{theorem}

Using this theorem it is not difficult to show that many groups
satisfy the EN-condition.
\begin{corollary}\label{solvthm}
\begin{enumerate}
\item Every solvable (finite) group satisfies the EN-condition. Moreover, every group
of order $n$ that contains a solvable subgroup of size at least $\sqrt n \log^2 n$
satisfies the EN-condition.
In particular every
finite group of odd order satisfies  this condition.
\item Every symmetric group $S_n$ (and every alternating group $A_n$) satisfies
the EN-condition.
\end{enumerate}
\end{corollary}

Estimating the size of the smallest possible basis for explicitly given sets is often
far harder.
Erd\H{o}s and Newman showed that any basis for the set of squares
$\{t^2 : t=1,\dotsc,n\}$ (which is a subset of $\{1,2, \ldots ,n^2\}$)
is of size at least $n^{2/3-o(1)}$
for large values of $n$, which is an improvement
over the trivial lower bound of $n^{1/2}$.
They constructed a small basis for the squares, of size
only $O\bigl(\frac{n}{\log^M n}\bigr)$ for any $M$.
Wooley \cite[Problem~2.8]{cite:aimprobs} asked
about powers other than the squares.
Whereas it is likely that any basis for the set of $d$-th powers
$\{t^d : t=1,\dotsc,n\}$ is of size $\Omega(n^{1-\epsilon})$ for every
$\epsilon>0$ and $d\geq 2$, we can report only a modest improvement
of the $n^{2/3-o(1)}$ lower bound of Erd\H{o}s and Newman
for large values of $d$.
\begin{theorem}\label{enimpr}
The set $\{t^d : t=1,\dotsc,n\}$ does not have a basis
of size $O(n^{3/4-\frac{1}{2\sqrt{d}}-\frac{1}{2(d-1)}-\epsilon})$
for any $\epsilon>0$.
\end{theorem}

The rest of this short paper is split into four sections. The first one
describes several constructions of small $k$-universal sets.
It is followed by two sections containing the results
on small bases that answer the Erd\H{o}s-Newman question, and the result on
the bases for powers of integers.
The last section contains some concluding remarks.

We will employ the following notation. For a set
$X$ we denote by $X^t$ the $t$-fold Cartesian product
$X\times\dotsc\times X$. The notations
$2^X$ and $\binom{X}{t}$ denote the family of all subsets
of $X$ and the family of all $t$-element subsets of $X$, respectively.
For the sake of clarity, throughout the paper (including the introduction)
we do not make any serious attempt to
optimize the absolute constants in our statements and proofs, and omit
all floor and ceiling signs whenever these are not crucial.

\section{Universal sets}
In this section we present several results about
small $k$-universal sets. Recall that for any group $G$ every such set should
contain at least $\frac{1}{2}|G|^{1-1/k}$ elements.
We start off with a simple probabilistic construction of $k$-universal
sets that are only logarithmic factor larger than this lower bound.
\begin{proof}[Proof of Theorem~\ref{randthm}]
Let $X$ be a subset such that $Z=XX$ has size at most $3|X|$.
Let $p=\left(\frac{\abs{X}}{2k^3\log\abs{X}}\right)^{-1/k}$.
If $p>1$, then $2k^3 \log {\abs{X}} > |X|$, implying that
$\frac{\abs{X}}{\log {\abs{X}}} < 36^k$
and therefore that $|X| \leq 36|X|^{1-1/k} \log^{1/k} |X|$.
In this case there is nothing to
prove, as $X$ itself is obviously $k$-universal for $X$.
Let $U$ be a random subset of $Z$ obtained by picking each element of $Z$,
randomly and independently, with probability $p$.
Fix any $k$-element set $S\subset X$.
For any $x\in X$ the set $xS$ is contained in $Z$. Note that if two
sets $xS$ and $x'S$ have a non-empty intersection, then
$x'= xs_1s_2^{-1}$ for some $s_1,s_2 \in S$. This shows that
each set $xS$ intersects fewer than $k^2$ other sets of this form.
So for every subset $X'\subset X, |X'|< |X|/k^2$ there is an element $x^* \in X$ such that
$x^*S$ is disjoint from $xS$ for all $x \in X'$.
Therefore the maximum subset $X'\subset X$ such that $\{xS\}_{x\in X'}$ are pairwise disjoint
contains at least $\abs{X'}\geq \frac{\abs{X}}{k^2}$ elements.  Fix such an $X'$.

For any $x\in X$ the probability that $xS\subset U$ is
$p^k=\frac{2k^3\log\abs{X}}{\abs{X}}$. Therefore the probability that there is no $x\in X'$
such that $xS\subset U$ is
\begin{equation*}
\left(1-\frac{2k^3\log\abs{X}}{\abs{X}}\right)^{\abs{X'}}\leq
\left(1-\frac{2k^3\log\abs{X}}{\abs{X}}\right)^{|X|/k^2} \leq
e^{-2k\log\abs{X}}= \frac{1}{\abs{X}^{2k}}
\end{equation*}
As there are no more than $\abs{X}^k$ $k$-element subset of $X$, the
probability that $U$ is not $k$-universal for $X$ is at most
$1/\abs{X}^k\leq 1/2$.

On the other hand, $\mathbb{E}\big[\abs{U}\big]=p|Z|$ and by Markov's inequality
\begin{equation*}
\Pr\Big[\abs{U}>3 p\abs{Z}\Big]= \Pr\Big[\abs{U}>3 \mathbb{E}\big[\abs{U}\big]\Big]
\leq \frac{1}{3}.
\end{equation*}
Therefore, with probability at least $1-1/2-1/3>0$, $\abs{U}\leq 3p\abs{Z}$ and it is
$k$-universal for $X$. In particular, such a $U$ exists.
The theorem now follows since $|Z| \leq 3|X|$ and $(2k^3)^{1/k}\leq 4$.
\end{proof}

\vspace{0.1cm}
\noindent
{\bf Remark.}\,
>From the above proof one can easily deduce that for every two subsets $X, X^*$ of
a finite group $G$ with $|X|=|X^*|$,
$G$ contains a $k$-universal set $U$ for $X$ of size at most
$$|U| \leq 12 \frac{|X^*X|\log^{1/k} |X|}{|X|^{1/k}}\,.$$
Thus, if a set $X$ does not grow significantly when multiplied by
some other set of the same cardinality,
this estimate gives a $k$-universal set for $X$
whose size is only slightly worse than the coresponding lower bound.

\begin{proof}[Proof of Theorem~\ref{constrthm} part \ref{cyclicconstr}]
If $\abs{G}\leq \exp(2^k)$, then it is easy to check that the existence of the desired
$k$-universal set follows from theorem~\ref{randthm}. Assume
that $\abs{G}\geq \exp(2^k)$. Let $p$ be a prime and let
$r=\frac{p^{k+1}-1}{p-1}$. We first construct, for any prime $p$, a
$k$-universal set of size $\frac{p^k-1}{p-1} \approx r^{1-1/k}$ in
$\Z/r\Z$. The construction
is motivated by Singer's theorem \cite{singerthm}.

Let $q=p^{k+1}$ and denote by $F_q$ the finite field of $q$ elements.
Let $\omega$ be a generator of $F_q^*$, and think
of $F_q$ as a vector space of dimension $k+1$ over $F_p$. Since
$\omega^r\in F_p$ and every element
of $F_q^*$ is a power of $\omega$, every $1$-dimensional
subspace of $F_q$ is of the form $w^t F_p$ with
$t\in \Z/r\Z$. Since $(p^{k+1}-1)/(p-1)=r$ is also the number of
$1$-dimensional
subspaces of $F_{q}$ the map
$\phi\colon t\mapsto \omega^t F_p$ is a bijection
between $\Z/r\Z$ and the set of $1$-dimensional subspaces.
Fix any basis of $F_{q}$ and consider a standard coordinate-wise
scalar product on $F_{q}$. For a $1$-dimensional
subspace $L$ of $F_{q}$ let $L^{\bot}$ be the
orthogonal complement of $L$, which is $k$-dimensional.
Since the map $L\mapsto L^{\bot}$ taking a $1$-dimensional
space to its orthogonal complement,
is a bijection, every
$k$-dimensional subspace of $F_q$ is of the form
$\omega^t (F_p)^{\bot}$ for a unique $t\in \Z/r\Z$.

Let $H$ be any subspace of $F_q$ of dimension $k$. Let
$X=\{\phi^{-1}(L) : L\subset H\}$ where $L$ ranges
over all $1$-dimensional subspaces of $H$. The set $X$ is
of size $\abs{X}=(p^k-1)/(p-1)$. Moreover, $X$ is $k$-universal
for $Z/r\Z$. Indeed if $L_1,\dotsc,L_k$ are any
$k$ $1$-dimensional subspaces of $F_q$, then their
span is contained in a $k$-dimensional subspace. Therefore
there is a $t$ such that $\omega^t L_i\subset H$ for $i=1,\dotsc,k$.

If we now think of $X$ not as a subset of $\Z/r\Z$
but of $\{1,\dotsc,r\}$, then $Y=X\cup\bigl(X+r\bigr)$
contains a translate of every $k$-element subset of
$[1,(p^{k+1}-1)/(p-1)]$. Therefore for every cyclic group $G$ of
size $\abs{G}\leq p^k\leq (p^{k+1}-1)/(p-1)$
there is a $k$-universal set of size at most $2(p^k-1)/(p-1)\leq 4p^{k-1}$.
By \cite[Theorem~1]{explicit_pnt} we know that
for every $x>1$ there is a prime $p$ satisfying $x\leq p\leq x (1+\frac{2}{\log x})$
and in particular, there is a prime $p$ such that
$\abs{G}^{1/k}\leq p\leq \abs{G}^{1/k}\bigl(1+\frac{2 k}{\log \abs{G}}\bigr)$.
Therefore every cyclic group $G$ satisfying $\abs{G}\geq \exp(2^k)$ contains a $k$-universal
set of size at most
$$4\abs{G}^{1-1/k}\left(1+\frac{2k}{\log \abs{G}}\right)^{k-1}\leq 4\exp(2k^2/2^k)\cdot
\abs{G}^{1-1/k}< 72\abs{G}^{1-1/k},$$
where here we used that $\exp(2k^2/2^k)< 18$ for all $k$, (with room to spare).
%(being largest when $k=3$).
\end{proof}

To prove the existence of small $k$-universal sets in any Abelian
group and in $S_n$ we need a more flexible concept than that of
a $k$-universal set. In order to induct on the size of the group,
one needs to be able to force some of the elements in a translate of a
$k$-element to be confined to a subset of $G$ of size much smaller
than $\abs{G}^{1-1/k}$. This prompts the following definition,
which is inspired by \cite{kozma_lev}.

A $k$-tuple of sets $U=(U_1,\dotsc,U_k)$ where
$U_i\subset G$ is said to be universal if for any $k$-tuple
$W=(w_1,\dotsc,w_k)\in G^k$ there is a $g\in G$ satisfying
$gw_i\in U_i$ for $i=1,\dotsc,k$. Equivalently, $(U_1,\dotsc,U_k)$ is a universal
$k$-tuple if and only if
$\{(u_1^{-1}u_2,\dotsc,u_{k-1}^{-1}u_k) : u_i\in U_i\}=G^{k-1}$.
Although we will not need this equivalence here, we include a short proof that it holds.
Suppose $\{(u_1^{-1}u_2,\dotsc,u_{k-1}^{-1}u_k) : u_i\in U_i\}=G^{k-1}$,
then for a given $(w_1,\dotsc,w_k)\in G^k$ we can solve the system of equations
\begin{align}
\label{cond1}
w_1^{-1}w_2&=u_1^{-1} u_2,\nonumber\\
\vdots&&&u_i\in U_i\qquad\text{ for }i=1,\dotsc,k.\\
w_{k-1}^{-1}w_k&=u_{k-1}^{-1}u_k,\nonumber
\end{align}
Let $g=w_1u_1^{-1}$. Since $w_iu_i^{-1}=w_{i+1}u_{i+1}^{-1}$,
by induction on $i$ it follows that $g=w_iu_i^{-1}$
for $i=1,\dotsc,k$. Hence $g^{-1}w_i=u_i$ is an element of $U_i$.
The reverse direction is similar.

If $(U_1,\dotsc,U_k)$ is a universal $k$-tuple of sets,
then $\bigcup_{i=1}^k X_i$ is a $k$-universal set. The greater flexibility
of universal $k$-tuples
comes at a price: construction of small universal $k$-tuples is more
involved than the construction of $k$-universal sets, and they are not
as small.
\begin{theorem}\label{cyclictuple}
For any cyclic group and
real numbers $1\leq s_1,\dotsc,s_k\leq\abs{G}$ satisfying
$\prod_{i=1}^k s_i=\abs{G}^{k-1}$ there is a
universal $k$-tuple $(U_1,\dotsc,U_k)$ satisfying
$\abs{U_i}\leq 8s_i$.
\end{theorem}

\noindent
{\em Proof.}\,
Let $t_i=\abs{G}/s_i$. By definition, it is easy to see
that $\prod_{i=1}^k t_i=|G|$. Since for every $x\geq 1$ there is a
non-negative integer $a$ for which $x/2\leq 2^a\leq x$,
we can select non-negative integers $p_1,\dotsc,p_k$
inductively one by one so that $\tfrac{1}{2}t_i\leq 2^{p_i}\leq 2 t_i$ and
\begin{equation}\label{sizecond}
\prod_{i=1}^r t_i\leq \prod_{i=1}^r 2^{p_i}\leq 2\prod_{i=1}^r t_i\qquad\text{for }r=1,\dotsc,k.
\end{equation}
Indeed, having selected $p_1$ through $p_r$ satisfying \eqref{sizecond}
we select $p_{r+1}$ so that it satisfies
$x/2\leq 2^{p_{r+1}}\leq x$ with $x=2t_{r+1}(\prod_{i=1}^r (t_i/2^{p_i})$.
Note that $ t_{r+1} \leq x \leq 2t_{r+1}$, since by induction  we already have that $1/2 \leq \prod_{i=1}^r (t_i/2^{p_i}) \leq 1$.
Therefore $t_{r+1}/2 \leq 2^{p_{r+1}}\leq 2t_{r+1}$.
Set $P=\sum_{i=1}^k p_i$, and
note that $\abs{G}=\prod_{i=1}^k t_i \leq 2^P\leq 2\prod_{i=1}^k t_i=2\abs{G}$. We will first construct
a universal $k$-tuple of sets $(Y_1,\dotsc,Y_k)$
for the group $\Z/2^P\Z$ satisfying $\abs{Y_i}=2^{P-p_i}$.

Every element of $\Z/2^P\Z$ can be written as a $P$-digit long
number in binary. The digits of such a number are naturally indexed
from $0$ to $P-1$ according to the power of $2$ that they represent.
The set $Y_1$ consists of the numbers whose
$p_1$ least significant digits are all zero. The set $Y_2$
is the set of numbers whose digits from position $p_1$ to
$p_1+p_2-1$ are all zero. In general, a number belongs to the
set $Y_i$ if all the digits from position $\sum_{j=1}^{i-1} p_i$
to $\sum_{j=1}^i p_i-1$ are zero. The $k$-tuple $(Y_1,\dotsc,Y_k)$
is universal. Indeed suppose we are given any $(v_1,\dotsc,v_k)\in (\Z/2^P\Z)^k$.
To find $g\in\Z/2^P\Z$ satisfying $g+v_i\in Y_i$ we proceed in stages.
First the condition $g+v_1\in Y_1$ determines the $p_1$ least significant
digits of $g$. Once these digits are fixed,
whatever the choices for the remaining digits are,
we will have $g+v_1\in Y_1$.
So, we can proceed to find the next $p_2$ digits from the condition
$g+v_2\in Y_2$, and so on until all the digits of $g$ are determined.

Let $Z_i=\{ n\in \{1,\dotsc,2^P+\abs{G}\} : n \bmod 2^P\in Y_i\}$.
As $\abs{G}\leq 2^P$, for every $(v_1,\dotsc,v_k)\in\{1,\dotsc,\abs{G}\}^k$
one can find $0\leq g<2^P$ for which $(g+v_i)\bmod 2^P\in Y_i$
for $i=1,\dotsc,k$. However, $0<g+v_i<2^P+\abs{G}$ implying
$g+v_i\in Z_i$. Thus the sets $U_i=Z_i\bmod \abs{G}$ together
form a universal $k$-tuple in $\Z/\abs{G}\Z\cong G$.
Their sizes are
$$\hspace{4.2cm} \abs{U_i}\leq \abs{Z_i}\leq 2\abs{Y_i}=
2\frac{2^P}{2^{p_i}}\leq 4\frac{\abs{G}}{2^{p_i}}\leq
8\frac{\abs{G}}{t_i}=8s_i.\hspace{4.2cm} \Box$$

Next we show how one can sometimes reduce the construction
of universal $k$-tuples for a group $G$ to the
construction of universal $k$-tuples for a subgroup $H$ of $G$.
The following definitions introduce a measure that estimates how small the
universal $k$-tuples in a given group can be.
For a group $G$ and real numbers $1\leq s_1,\dotsc,s_k\leq \abs{G}$
satisfying $\prod_{i=1}^k s_i=\abs{G}^{k-1}$ let
\begin{equation*}
r_k(G;s_1,\dotsc,s_k)=\min\left\{\frac{\abs{U_1}}{s_1}+\dotsb+\frac{\abs{U_k}}{s_k} :
(U_1,\dotsc,U_k)\text{ is a universal $k$-tuple}\right\}
\end{equation*}
and define $r_k(G)$ by
\begin{equation*}
r_k(G)=\sup\left\{r_k(G;s_1,\dotsc,s_k) : 1\leq s_1,\dotsc,s_i\leq \abs{G}\text{ and }\prod_{i=1}^k s_i=\abs{G}^{k-1}\right\}.
\end{equation*}
By definition, for any universal $k$-tuple $U=(U_1, \ldots, U_k), U_i\subset G$ the set $\cup_{i=1}^kU_i$ is
$k$-universal. Therefore, taking all $s_i=\abs{G}^{1-1/k}$ we have that
any finite group admits a $k$-universal set
of size at most $r_k(G)\abs{G}^{1-1/k}$. The theorem above implies that
$r_k(G)\leq 8k$ for any cyclic group. This is worse than the estimate provided by part
\ref{cyclicconstr} of Theorem
\ref{constrthm}, due to the dependence on~$k$, but this dependence is unavoidable
for this approach. Indeed, it is easy to see that for every universal $k$-tuple
$\prod \abs{U_i}\geq \abs{G}^{k-1}$ and thus the arithmetic-geometric means
inequality implies that $r_k(G)\geq k$
for any group~$G$.

The heart of the inductive construction of small $k$-universal sets
for non-cyclic groups is the following lemma.
\begin{lemma}\label{subgrouplem}
Let $H$ be a subgroup of a finite group $G$ such that
$\abs{H}\geq \abs{G}^{1-1/k}$. Then
\begin{equation*}
r_k(G)\leq r_k(H).
\end{equation*}
\end{lemma}
\begin{proof}
Our aim is to show that $r_k(G;s_1,\dotsc,s_k)\leq r_k(H)$ for
any choice $1\leq s_1,\dotsc,s_k\leq\abs{G}$ satisfying
$\prod_{i=1}^k s_i=\abs{G}^{k-1}$. Fix any such choice.
We claim that there is at most one index $j$ for which
$s_j<\abs{G}/\abs{H}$. Indeed, had there been two such indices
then it would follow that $\abs{G}^{k-1}=\prod_{i=1}^k s_i<\left(\frac{\abs{G}}{\abs{H}}\right)^2\abs{G}^{k-2}$ contradicting $\abs{H}^2\geq
\abs{H}^{\frac{k}{k-1}}\geq \abs{G}$. Moreover there has to be
an index $j$ for which $s_j\leq \abs{H}$ or else
$\abs{G}^{k-1}=\prod_{i=1}^k s_i>\abs{H}^k\geq \abs{G}^{k-1}$.
In short: there is an index $j$ such that
$s_j\leq \abs{H}$, but for all $i\neq j$ we
have $s_j\geq \abs{G}/\abs{H}$.

Set $t_i=s_i\abs{H}/\abs{G}$ for
$i\neq j$ and $t_j=s_j$. Then $\prod_{i=1}^k t_i=\abs{H}^{k-1}$
and $1\leq t_1,\dotsc,t_k\leq \abs{H}$.
Let $(U_1,\dotsc,U_k)$ be a universal $k$-tuple in $H$ satisfying
\begin{equation*}
\frac{\abs{U_1}}{t_1}+\dotsb+\frac{\abs{U_k}}{t_k}\leq r_k(H).
\end{equation*}
Let $T_1$ be a set of representatives of the right cosets of $H$ in $G$ and
and let $T_2$ be the set of inverses of elements of $T_1$.
Then $|T_1|=|T_2|=\abs{G}/\abs{H}$.
Let
\begin{equation*}
Y_i=\begin{cases}
U_i T_1,&\text{if }i \neq j,\\
U_i,&\text{if }i=j.
\end{cases}
\end{equation*}
Then, as we show below, $(Y_1,\dotsc,Y_k)$ is a universal $k$-tuple satisfying
$$\sum_i\frac{|Y_i|}{s_i}=\sum_{i\not =j}\frac{|Y_i|}{t_i(|G|/|H|)} +\frac{|U_j|}{t_j} \leq
\sum_{i\not =j}\frac{|U_i||T_1|}{t_i(|G|/|H|)} +\frac{|U_j|}{t_j}=
\sum_i\frac{|U_i|}{t_i},$$
which shows $r_k(G;s_1,\dotsc,s_k)\leq r_k(H)$. To verify that $(Y_1,\dotsc,Y_k)$ is universal
consider an arbitrary sequence $(w_1, w_2, \ldots ,w_k) \in G^k$.  For each
$i \neq j$, let $t_i \in T_2$ satisfy $w_j^{-1}w_i t_i \in H$, (such a $t_i$ exists
by the defintion of $T_2$).
Since $(U_1, \ldots ,U_k)$ is a universal $k$-tuple for $H$, and as
$w_j^{-1}w_j=1$ is clearly in $H$, there is an
$h \in H$ such that  for every $i \neq j$,  $hw_j^{-1} w_i t_i \in U_i$
and $hw_j^{-1}w_j \in U_j$. Define $g=hw_j^{-1}$. Then
$gw_i \in U_i t_i^{-1} \subset Y_i$ for all $i \neq j$, and
$gw_j \in U_j=Y_j$.  Therefore $(Y_1,\dotsc,Y_k)$ is universal, as needed.
\iffalse
it suffices, by (\ref{cond1}), to show that for every $(w_1,\dotsc,w_k)\in G^k$
the system of equations
\begin{align*}
w_i^{-1}w_{i+1}&=(u_ig_i)^{-1} u_{i+1},&&g_i\in T_l&&\text{ for }i=1,\dotsc,j-1,\\
w_i^{-1}w_{i+1}&=u_i^{-1}(u_{i+1}g_{i+1})&&g_{i+1}\in T_r&&\text{ for }i=j,\dotsc,k-1
\end{align*}
has a solution with $u_i\in U_i$.
However, by definition of $T_1, T_2$, we can choose $g_i \in T_2$
such that $g_i(w_i^{-1}w_{i+1}) \in H$ for $1\leq i\leq j-1$
and choose $g_{i+1} \in T_1$ such that $g_{i+1}^{-1}(w_i^{-1}w_{i+1})\in H$
for $j\leq i\leq k-1$. Now the existence of the solution
follows from the assumption that $(U_1,\dotsc,U_k)$ is a universal $k$-tuple.
\fi
\end{proof}

To construct small $k$-universal sets, it therefore suffices
to find a large subgroup for which one can construct small universal $k$-tuples.
We do it for symmetric and Abelian groups.
\begin{proof}[Proof of parts \ref{symconstr} and \ref{abelconstr} of Theorem~\ref{constrthm}]
b)\, For $n\geq 3k$ it is easy to check that
$\abs{S_{n-1}}=(n-1)! \geq (n!)^{1-1/k}=\abs{S_n}^{1-1/k}$.
Therefore to bound $r_k(S_n)$ for all $n$,
by the preceding lemma, it suffices to bound $r_k(S_n)$ for $n\leq 3k-1$.
Since $r_k(G)\leq k\abs{G}$ for any group, it follows that
for the symmetric group $r_k(S_n) \leq k(3k-1)!$.
This implies that $G=S_n$ has a $k$-universal set of size
$kr_k(S_n)|G|^{1-1/k} <(3k+1)!|G|^{1-1/k}$.

c)\,  We will show that $r_k(G)\leq 8^{k-1}$ for any Abelian group $G$.
Let $G=G_1\times\dotsb\times G_t$ be an arbitrary Abelian group
written as a direct product of cyclic groups. If $t\geq k$, then
at least one of these cyclic groups has order
$\leq \abs{G}^{1/k}$, and the product of the remaining groups
is a subgroup of $G$ of order $\geq \abs{G}^{1-1/k}$. Therefore
by lemma~\ref{subgrouplem} it suffices to consider only
the Abelian groups with $t \leq k-1$ cyclic factors.
Let $s_1,\dotsc,s_k$ satisfying $\prod_{i=1}^k s_i=\abs{G}^{k-1}$ be given.
Let $(U_{r,1},\dotsc,U_{r,k})$ be a universal $k$-tuple
for $G_r$ with $\abs{U_{r,i}}\leq 8 s_i^{\log_{\abs{G}}\abs{G_r}}$
which exists by theorem~\ref{cyclictuple} since
$\prod_{i=1}^k s_i^{\log_{\abs{G}}\abs{G_r}}=\abs{G_r}^{k-1}$. Let
$U_i=U_{1,i}\times\dotsb\times U_{t,i}$ be the Cartesian product.
Then $(U_1,\dotsc,U_k)$ is a universal $k$-tuple for $G$
satisfying
$\abs{U_i}\leq \prod_{r=1}^t 8 s_i^{\log_{\abs{G}}\abs{G_r}}=8^t s_i$.
\end{proof}

\section{Small bases}
In this section we show how to use universal sets  to construct small bases
for subsets of size $\leq \sqrt{\abs{G}}$ of a group $G$.
\begin{lemma}
For every $X\subset G$ satisfying $\abs{X}\geq \sqrt{\abs{G}}\log^2\abs{G}$
there is a subset $Y\subset G$ of size $\abs{Y}\leq s$, where
$s=\frac{\sqrt{\abs{G}}}{\log \abs{G}}$
so that $YX=G$.
\end{lemma}
\begin{proof}
Let $Y=\{y_1,\dotsc,y_s\}$ be a random set of $s$ (not necessarily distinct) elements,
where each $y_i$ is chosen uniformly at random from $G$, and all choices are independent.
Clearly an element $g \in G$ does not belong to $YX$ if and only if $Y$ is disjoint from
the set $\{gx^{-1}~|~x \in X\}$. Thus the probability that $g\in G$ does not belong to $YX$ is
$\bigl(1-\frac{\abs{X}}{\abs{G}}\bigr)^s$.
The expected number of elements of $G$ that do not belong to $YX$ is
\begin{equation*}
\abs{G}\left(1-\frac{\abs{X}}{\abs{G}}\right)^s\leq \abs{G}
\left(1-\frac{\log^2\abs{G}}{\sqrt{\abs{G}}}\right)^%
{\sqrt{\abs{G}}/\log\abs{G}}<1
\end{equation*}
implying that there is a choice of $Y$ for which $G=YX$.
\end{proof}

With the lemma in our toolbox we are just a step away from
showing that the existence of moderately sized non-doubling sets are all what is
needed to construct small bases.
\begin{proof}[Proof of Theorem~\ref{nondoublingENthm}]
Let $G$ be a group with $n$ elements containing a non-doubling set $X$ of size
$\sqrt{n}\log^2 n \leq \abs{X}\leq\sqrt{n}\log^{10}n$. By the above lemma $G$ contains a set $Y$
of size $s=\frac{\sqrt{n}}{\log n}$ such that $YX=G$.
Suppose $A\subset G$,
$\abs{A}\leq \sqrt{n}$. Partition $A$ into disjoint subsets $A_1,\dotsc,A_s$,
where $A_i$ are all members of $A$ that lie in $y_iX$ and do not lie
in any of the previous set $y_jX$ for $j<i$. Split each $A_i$ into
$\lceil \abs{A_i}/k\rceil$ pairwise disjoint sets $T_{i,j}$, each of size
at most $k$, where $k=\frac{\log n}{30\log\log n}$. Let $U$ be a $k$-universal set
for $X$ of size at most
\begin{eqnarray*}
|U| &\leq& 36|X|^{1-1/k}\log^{1/k}|X| \leq 36\Big(\sqrt{n}\log^{10}n\Big)^{1-1/k}\log^{1/k} n\\
&<& \frac{\sqrt{n}\log^{10}n}{n^{1/2k}}=\frac{\sqrt{n}}{\log^5 n},
\end{eqnarray*}
which exists by Theorem~\ref{randthm}. For each set $T_{i,j}$ defined above, let $g_{i,j}$ be
an element of $G$ so that $T_{i,j}\subset g_{i,j}U$. The existence
of $g_{i,j}$ follows from the fact that $U$ is $k$-universal for $X$
and the fact that $A_i$ is contained in the shift $y_iX$ of $X$. Note that
the number of elements $g_{i,j}$ is
\begin{equation*}
\sum_{i=1}^s \lceil \abs{A_i}/k\rceil\leq \frac{A}{k}+s \leq
\frac{\sqrt{n}}{k}+\frac{\sqrt{n}}{\log n}\leq
40\frac{\sqrt{n}\log\log n}{\log n}.
\end{equation*}
Finally, to complete the proof, define $B=U\cup\{g_{i,j}: 1\leq i\leq s,\ 1\leq j\leq \lceil\abs{A_i}/k\rceil\}$ and
note that $A\subset BB$ and $|B| \leq |U|+40\frac{\sqrt{n}\log\log n}{\log n}<50\frac{\sqrt{n}\log\log n}{\log n}$.
\end{proof}

\noindent\textbf{Remark: }The proof actually gives a somewhat stronger result than stated
in Theorem~\ref{nondoublingENthm}. Call a set $X\subset G$ non-expanding if there is a set
$X^*\subset G$ so that $\abs{X}=\abs{X^*}$ and $\abs{X^*X}\leq\abs{X}\log \abs{G}$.
Trivially, every non-doubling set is non-expanding.
Combining the estimate on the size of $k$-universal sets
from the remark following the proof of Theorem \ref{randthm}
together with the above arguments gives
the following
\begin{theorem}
If $G$ contains a non-expanding $X$ satisfying $\sqrt{\abs{G}}\log^2\abs{G}\leq
\abs{X}\leq\sqrt{\abs{G}}\log^{10}\abs{G}$,
then $G$ satisfies the EN-condition.
\end{theorem}

\begin{lemma}
\label{l33}
For every finite solvable group $G$ of order $m$ and every $x$ satisfying
$1 <x  \leq m$ there is a non-doubling  subset $X \subset G$ satisfying
$x \leq |X| \leq 2x$.
\end{lemma}

\begin{proof}
Let $G$ be a finite solvable group. Then there is a normal sequence
\begin{equation*}
\{1\}=G_k \subset G_{k-1} \subset \ldots \subset G_1 \subset G_0=G,
\end{equation*}
where each $G_{i+1}$ is a normal subgroup of $G_i$ and all the
quotients $G_i/G_{i+1}$ are cyclic. Let $i$ be the minimum index so
that $\abs{G_{i+1}}<x$. Let coset $hG_{i+1}$
be a generator of $G_i/G_{i+1}$, put
$t=\lceil \frac{x}{\abs{G_{i+1}}}\rceil$
and define $X=h^0G_{i+1}\cup h^1 G_{i+1}\cup\dotsb\cup h^{t-1}G_{i+1}$. Then
$X$ is of size at least $x$ and at most $2x$, and as
$XX=h^0G_{i+1}\cup h^1 G_{i+1}\cup\dotsb\cup h^{2t-2}G_{i+1}$
it is non-doubling.
\end{proof}

\noindent
Having finished all the necessary preparations, we can now prove
Corollary \ref{solvthm}

%show that
%subsets of symmetric or solvable group $G$ with at most $\sqrt{\abs{G}}$ elements have
%bases of size $50 \sqrt{\abs{G}}\log\log\abs{G}/\log\abs{G}$.

\begin{proof}[Proof of Corollary~\ref{solvthm}]
a)\, If a finite group $G$ of order $n$ contains a solvable subgroup
of size $m \geq \sqrt n \log^2 n$ then, by Lemma \ref{l33} with
$x = \sqrt n \log^2 n$, this subgroup
contains a non-doubling set of  size at least $x=\sqrt n \log^2 n$
and at most $2x$.
Thus, by Theorem~\ref{nondoublingENthm}, $G$ satisfies the EN-condition.
In particular, every finite solvable group  satisfies this condition, and
the assertion of part (a) follows, as every group of odd order is solvable,
by the Feit-Thompson theorem \cite{feit_thompson}.

b)\, Clearly $S_n$ contains a subgroup isomorphic to $S_m$ for all $m<n$. Since, all ratios
$$\abs{S_{m+1}}/\abs{S_m}=m+1\leq n < \log^2 (n!)=\log^2 \abs{S_n},$$
it is easy to see that $G=S_n$ contains a subgroup whose size is between  $\sqrt{\abs{G}}\log^2\abs{G}$
and $\sqrt{\abs{G}}\log^4\abs{G}$. As every subgroup is non-doubling, the
result now follows from Theorem~\ref{nondoublingENthm}.
\end{proof}

\noindent\textbf{Remark: } Corollary ~\ref{solvthm} can be used to show that many additional
finite groups satisfy the EN condition, as they contain large solvable subgroups. In
particular, this holds for all linear groups, see \cite{Ma}. It seems plausible that
in fact every finite group satisfies the EN-condition.

\section{Bases for powers}
To show that there is no small basis for the set of $d$'th powers, we
shall use estimates on the number of representations of a number as a
sum of several $d$'th powers. We let $P_d(n)=\{t^d : t=1,\dotsc,n\}$ be
the set of the first $n$ $d$'th powers.
\begin{lemma}
If the equation
\begin{equation}\label{sumpweq}
x_1^d-x_2^d+\dotsb+(-1)^{k+1}x_k^d=y,
\end{equation}
has fewer than
$O(B^{\epsilon})$ solutions in distinct positive integers
$x_1,\dotsc,x_k$ not exceeding $B$,
then $\abs{A}=\Omega(n^{1-\frac{1+\epsilon}{k+1}})$
for any set $A$ of positive integers for which $P_d(n)\subset A+A$.
\end{lemma}
\begin{proof}
Suppose $A$ is an $m$-element set satisfying $A+A\supset P_d(n)$.
Let $G$ be a graph on the vertex set $A$ in which for each
element $p \in P_d(n)$ we choose the lexicographically first pair
of elements
$a,a'\in A$ whose sum is $p$ and
connect them by an edge.
Clearly
$G$ contains exactly $n$ edges. Let
$a,a'\in A$ be any two vertices of the graph. Let $y=a+(-1)^{k-1}a'$.
Any path with distinct edges of length $k$ between $a$ and $a'$ gives a rise to a solution
of \eqref{sumpweq}. Indeed if $a=a_0,a_1,\dotsc,a_k=a'$, then
$y=(a_0+a_1)-(a_1+a_2)+\dotsb+(-1)^{k-1}(a_{k-1}+a_k)$
is an alternating sum of $d$'th powers. Since each path between $a$ and $a'$
corresponds to a different
solution of \eqref{sumpweq}, it follows that every pair
of vertices in $G$ are connected by no more than
$O(n^{\epsilon})$ such paths of length~$k$. As every graph on $m$
vertices and $n$ edges contains a non-empty subgraph with minimum degree
at least $n/m$, the graph $G$ contains such a subgraph\footnote{Indeed, removing
vertices of degree less than $n/m$ one by one increases the average degree of
the remaining graph. Therefore the process terminates with a non-empty graph of minimum
degree $\geq n/m$.}, which we
denote by $G'$. Let $a\in V(G')$ be any vertex of that graph.
Then there are more than $(\frac{n}{m}-k)^k$ paths of length $k$ consisting of
distinct edges originating
from $a$. By the pigeonhole
principle a fraction of $1/m$ of them has the same endpoint. Thus
$(\frac{n}{m}-k)^k \leq O(m n^{\epsilon})$ and the lemma follows.
\end{proof}

To give a lower bound on the size of a basis for $P_d(n)$ it therefore
suffices to give good estimates on the number of solutions to
\eqref{sumpweq}. Heath-Brown \cite[Theorem~13]{heath_brown}
proved that for $y\leq B^d$ the number of solutions
to $x_1^d+x_2^d+x_3^d=y$ in positive integers $x_1,x_2,x_3$ is
$O(B^{\frac{2}{\sqrt{d}}+\frac{2}{d-1}+o(1)})$.
Inspection of the proof shows that it applies verbatim to yield
the same bound on
the number of solutions to
$(\pm 1)x_1^d+(\pm 1)x_2^d+(\pm 1)x_3^d=y$ in distinct positive
integers $x_1,x_2,x_3$ not exceeding~$B$. This establishes
theorem~\ref{enimpr}. \vspace{1ex}

\section{Concluding remarks}
The questions we consider in this paper are all special cases of the following
general problem:
\begin{problem}
Given a set $X$, a group $G$ that acts on $X$ and a family
$\mathcal{F}\subset 2^X$, find the minimum possible size of
a set $U$ such that for every $S\in\mathcal{F}$ there is
a $g\in G$ for which $gS\subset U$.
\end{problem}

\noindent
For example, Theorem \ref{constrthm} treats the case
when $X=G$, $\mathcal{F}=\binom{G}{k}$ and $G$ acts on $X$ by left multiplication.
Bourgain's arithmetic version of Kakeya problem is essentially the
case $X=\Z/p\Z$ where $\mathcal{F}$ is the family of all $k$-term arithmetic progressions.

Another important special case of the above problem deals with
universal graphs. For a family
$\mathcal{H}$ of graphs, a graph $\Gamma$ is
$\mathcal{H}$-universal if, for each $H\in \mathcal{H}$, the
graph $\Gamma$ contains a subgraph isomorphic to $H$.
In our language it is the universal set problem
for $G=S_n$ and $X=\binom{[n]}{2}$ with $G$ acting on $X$ by permuting the elements
of $[n]$.
The universal graph problem, which deals with the minimum possible number of edges in a
graph that is universal for a given family,
has been studied extensively for many classes of graphs $\mathcal{H}$. For instance,
if $\mathcal{H}$ is the
family of all $k$-edge graphs, then the smallest $\mathcal{H}$-universal
graph is of size $\Theta(k^2/\log^2 k)$ \cite{alon_asodi}.  The problem
has also been studied for other families of graphs, such as graphs of bounded
degree, trees or planar graphs. The interested reader is referred to
\cite{alon_asodi}, \cite{ac} and the references therein.

Besides the Kakeya conjecture some other analogues of the universal set problem for
infinite
groups have been considered as well.
For example, Haight \cite{haight_semigroup} constructed a set $U$ of real
numbers of zero Lebesgue measure that contains a translate of every
countable set. The results in this paper can be adapted to construct
a $U\subset \R$ of Minkowski dimension $1-1/k$ containing a translate
of every $k$-element set.

The universal set problem is certainly hard in general.
It would be interesting to find suitable general conditions
under which it can be solved.\vspace{1.5ex}

\noindent\textbf{Acknowledgments.} We thank Roger Heath-Brown,
Per Salberger and Aner Shalev for helpful discussions.

\end{document}